\documentclass[11pt]{amsart}
\usepackage{amsmath,amssymb,amsbsy,amsthm,latexsym,
         amsopn,amstext,amsxtra,euscript,amscd,amsthm, mathtools, dsfont}
        
\usepackage{tikz}
\usetikzlibrary{arrows.meta}
\usepackage{fullpage,amssymb,amsmath}
\usepackage{enumerate}
\usepackage[shortlabels]{enumitem} 
\usepackage{comment}
\usepackage{hyperref} 
\usepackage[capitalise]{cleveref}
    \Crefname{proposition}{Proposition}{Propositions}
    \Crefname{lemma}{Lemma}{Lemmas}
    \Crefname{remark}{Remark}{Remarks}
    \Crefname{corollary}{Corollary}{Corollaries}
    \Crefname{definition}{Definition}{Definitions}
\usepackage{url}
\usepackage{xcolor}

\usepackage{bbm}
\usepackage{mathrsfs} 

\numberwithin{equation}{section}

\newtheorem*{corollary*}{Corollary}
\newtheorem*{theorem*}{Theorem}
\newtheorem{theorem}{Theorem}[section]

\newtheorem{lemma}[theorem]{Lemma}

\theoremstyle{definition}

\newcommand{\dist}{\operatorname{dist}}

\newcommand{\D}{\mathbb D}
\newcommand{\T}{\mathbb T}
\newcommand{\C}{\mathbb C}
\newcommand{\eps}{\varepsilon}
\newcommand{\om}{\omega}

\newcommand{\calL}{\mathcal L}

\title{A proof of conjectures of Esterle and Ransford on negative powers of contractions}
\author{William Verreault}
\date{}

\address{Department of Mathematics \\
University of Toronto   \\ 
Toronto, ON\\
Canada}
\email{william.verreault@utoronto.ca}

\begin{document}

\vspace*{-4mm}

\begin{abstract}
Building on work of Ransford, we prove that whenever $E$ is a closed subset of the unit circle of Lebesgue measure zero, there exists a positive sequence $u_n\to\infty$ such that if $T$ is a contraction on a Hilbert space with $\sigma(T)\subset E$ and $\|T^{-n}\|=O(u_n)$, then $T$ is unitary.
This confirms conjectures of Esterle and Ransford. Our main new idea is a spikes-in-collars principle for positive subharmonic functions.
\end{abstract}

\maketitle

\section{Introduction}
Let $T$ be a bounded linear operator on a complex Hilbert space, and write
$\sigma(T)$ for its spectrum. We say that $T$ is a contraction if
$\|T\|\le1$.  
We establish the following conjecture of Esterle \cite[Conjecture 1.3]{Ransford}.

\begin{theorem}[Esterle's conjecture]\label{thm:esterle}
Let $E\subset\T$ be closed and have Lebesgue measure zero.  Then there exists
a positive sequence $u_n\to\infty$ such that if $T$ is a contraction on a
Hilbert space, $\sigma(T)\subset E$, and
$
        \|T^{-n}\|=O(u_n)
$,
then $T$ is unitary.
\end{theorem}

The problem belongs to a line of work connecting negative powers of contractions, spectral synthesis phenomena, and uniqueness properties of closed subsets of the circle.
Ransford proved this under the additional finite-rank assumption
$\operatorname{rank}(I-T^*T)<\infty$ \cite{Ransford}, which entered through a determinant reduction.
Moreover, \cref{thm:esterle} is known to be sharp, in the sense that one cannot replace the measure zero and Hilbert space assumptions \cite{Ransford}. Weaker versions of this theorem and related results were known for various sets $E$ \cite{Esterle1, Esterle2, Kellay, Zarrabi}. See \cite{Ransford} for more information on the history of this problem.

We remove the finite-rank hypothesis of Ransford with new potential-theoretic input which may be of independent interest. 
In finite dimensions, one
can reduce the problem to a scalar singular-inner
function whose singular measure is supported on $E$.  Instead of trying to replace the determinant in infinite dimensions, we use the subharmonic function
$$
        w_\Theta(z)=\log\|\Theta(z)^{-1}\|.
$$

Let $E\subset\T$ be closed and have Lebesgue measure zero.  We say that a subharmonic function $w$ on $\D$ is \emph{carried by $E$} if
$w$ is locally bounded in a neighbourhood of $\overline\D\setminus E$ and
has continuous boundary value zero on $\T\setminus E$. We show that such a function must have sufficiently large spikes somewhere in
shrinking collars over that set.

\begin{theorem}\label{thm:intro-collar-spike}
Let $E\subset\mathbb T$ be closed and have Lebesgue measure zero. Then
there exists a continuous nonincreasing function
$
        L:(0,1]\to[1,\infty)
$
such that
$
        L(t)\to\infty$ and
        $tL(t)\to 0$ as $t\to 0$, 
with the following property.

Suppose $w$ is a nonzero,
nonnegative subharmonic function carried by $E$. Then for any $A>0$ and
$\eta>0$, there exists a point $z=re^{i\theta}\in\mathbb D$ such that
$$
        0<1-r<\eta,
        \qquad
        \dist(e^{i\theta},E)<\eta,
$$
and
$$
        w(z)\ge A L(1-r).
$$
\end{theorem}

\subsection*{Remark}
The author had the idea for this proof after a presentation of Thomas Ransford on \cite{Ransford} at the \textit{Conference on Classical Analysis in Memory of Paul Koosis}. Ransford independently came up with another proof \cite{Ransford2} following this event, based on a method suggested by Fedja Nazarov. 

\section{Proof of \cref{thm:intro-collar-spike}} \label{sec:potential}
For $\eta>0$, put
$$
        E_\eta=\{\zeta\in\T:\dist(\zeta,E)\le \eta\},
$$
and define the collar over $E_\eta$ by
$$
        K_\eta=\{re^{i\theta}\in\D:1-\eta\leq r<1,\ e^{i\theta}\in E_\eta\}.
$$
Let $\Omega_\eta$ be the connected component of $\D\setminus K_\eta$
containing the origin, and put
$$
        \Sigma_\eta=\partial\Omega_\eta\cap\D.
$$
Thus $\Sigma_\eta$ is the inner boundary of the removed collar.  We write
$\om_a^{\Omega_\eta}$ for harmonic measure in $\Omega_\eta$ from
$a\in\Omega_\eta$.

\begin{figure}[ht]
\centering
\begin{tikzpicture}[
    scale=1.05,
    every node/.style={font=\small},
    line cap=round,
    line join=round
]
    \def\R{1.95}   
    \def\ri{1.56}  
    \def\a{36}     
    \def\b{76}     

    \fill[gray!14]
        (\a:\ri) arc[start angle=\a,end angle=\b,radius=\ri]
        -- (\b:\R) arc[start angle=\b,end angle=\a,radius=\R]
        -- cycle;

    \draw[line width=0.45pt]
        (\b:\R) arc[start angle=\b,end angle=\a+360,radius=\R];

    \draw[dashed, dash pattern=on 3pt off 2pt, line width=0.75pt]
        (\a:\R) arc[start angle=\a,end angle=\b,radius=\R];

    \draw[line width=0.45pt] (\a:\ri) -- (\a:\R);
    \draw[line width=1.00pt]
        (\a:\ri) arc[start angle=\a,end angle=\b,radius=\ri];
    \draw[line width=0.45pt] (\b:\ri) -- (\b:\R);

    \fill (0,0) circle[radius=0.75pt];
    \node[below=2pt] at (0,0) {$0$};

    \node at (60:2.18) {$E_\eta$};
    \node at (56:1.34) {$\Sigma_\eta$};
    \node at (0.72,-0.24) {$\Omega_\eta$};
\end{tikzpicture}
\caption{The sets $E_\eta$ (dashed), $K_\eta$ (shaded region), $\Omega_\eta$, and $\Sigma_\eta$ (bold) used in the proof of \cref{thm:intro-collar-spike}, in the case of one connected component.}
\end{figure}
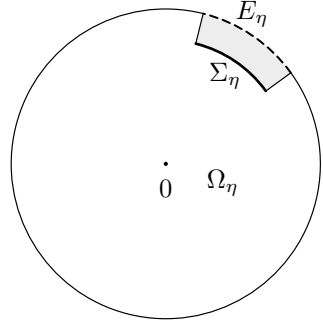

\subsection{Harmonic measure estimates}
For $0<R<1$, define
$$
        p_R(\eta)=
        \sup_{|a|\le R}\om_a^{\Omega_\eta}(\Sigma_\eta),
        \qquad \eta<1-R,
$$
and for fixed $\eta>0$ and $0<R<1-\eta$, define
$$
        q_{R,\eta}(s)=
        \sup_{|a|\le R}
        \om_a^{\Omega_\eta}\bigl(\Sigma_\eta\cap\{1-|z|\le s\}\bigr).
$$

\begin{lemma}\label[lemma]{lem:pR}
For each $0<R<1$,
$
        p_R(\eta)\to 0$ as $
        \eta\to 0.
$
Moreover, for each fixed $\eta>0$ and $0<R<1-\eta$,
$
        q_{R,\eta}(s)\to 0$ as $
        s\to 0$.
\end{lemma}

\begin{proof}
Write $u_\eta(a)=\om_a^{\Omega_\eta}(\Sigma_{\eta})$.
If $0<\eta'<\eta$, then $K_{\eta'}\subset K_\eta$, and hence
$\Omega_\eta\subset\Omega_{\eta'}$. Thus $u_{\eta'}$ is harmonic on $\Omega_{\eta}$. 
On $\Sigma_\eta$, we have
$
        0\le u_{\eta'}\le 1=u_\eta,
$
while on the outer boundary $\mathbb T\setminus E_\eta$, both functions have
boundary value $0$.  By the maximum principle,
$u_{\eta'}\le u_\eta$
        in $\Omega_\eta.$
In particular, $u_\eta(0)$ decreases as $\eta\to 0$.

By Harnack's principle, $u_\eta$ converges locally uniformly as
$\eta\to 0$ to a bounded harmonic function $u$ on $\D$.  Let
$J$ be a closed arc contained in $\T\setminus E$.  For all sufficiently
small $\eta$, the collar $K_\eta$ is disjoint from a neighbourhood of
$J$, and $u_\eta$ has boundary value $0$ on $J$.  Since
$0\le u\le u_\eta$, it follows that $u$ has boundary value $0$ on
$J$.  As $J\subset\T\setminus E$ was arbitrary and $|E|=0$, we deduce that $u\equiv0$ and therefore that $u_\eta(0)\to 0$.
Harnack's inequality gives the desired uniform convergence.

For the second assertion, for fixed $\eta$, the set $E_\eta$ is a finite union of arcs.  Hence
$\Omega_\eta$ is a finitely connected domain with piecewise smooth boundary,
and the sets $\Sigma_\eta\cap\{1-|z|\le s\}$ decrease, as $s\to 0$, to a finite set of boundary
points.  Harmonic measure has no atoms at boundary points of such a domain, so
$$
       \om_0^{\Omega_\eta}(\Sigma_\eta\cap\{1-|z|\le s\})\xrightarrow[]{s\to 0} 0.
$$
Another application of Harnack's inequality on compact subsets of
$\Omega_\eta$ gives
the desired uniformity.
\end{proof}

Let $0<R_1<R_2<\cdots\uparrow1$.  We now choose a slowly increasing weight.

\begin{lemma}\label[lemma]{lem:weighted-collar}
There exist a decreasing sequence $\eta_j\to0$ and a continuous
nonincreasing function $L:(0,1]\to[1,\infty)$ such that
$ L(t)\to\infty$ and $tL(t)\to0$ as $t\to 0$,
and, for every fixed $m$,
$$
        \sup_{|a|\le R_m}
        \int_{\Sigma_{\eta_j}} L(1-|z|)\,
        d\omega_a^{\Omega_{\eta_j}}(z)
        \xrightarrow[]{j\to\infty}0 .
$$
\end{lemma}

\begin{proof}
We first construct a step function with the desired property, and then smooth
it slightly.
Choose $\eta_j\to 0$ recursively.  At stage $j$, impose the finitely
many conditions
\begin{align}
        j\,p_{R_m}(\eta_j)&\le 2^{-j-m},
        && m\le j, \label{eq:pj-condition}\\
        q_{R_m,\eta_i}(\eta_j)&\le 2^{-j-i-m},
        &&i<j,\ m\le i, \label{eq:qj-condition}
\end{align}
and also arrange $\eta_j <1-R_j$ and $j\eta_j\to0$. This is possible by \cref{lem:pR}.

Define a step function $L_0:(0,1]\to[1,\infty)$ by $L_0(t)=j$ for $\eta_{j+1}<t\le\eta_j$,and extend it as a bounded positive function on $[\eta_1,1]$.  Then
$
        L_0(t)\to\infty$ and 
        $
        tL_0(t)\to0
$.

Fix $m$, $j\ge m$, and $|a|\le R_m$.  Let $\nu_{j,a}$ be the image of
$\omega_a^{\Omega_{\eta_j}}|_{\Sigma_{\eta_j}}$ under the map
$z\mapsto 1-|z|$.  Thus
$$
        \nu_{j,a}((0,s])
        =
        \omega_a^{\Omega_{\eta_j}}
        \bigl(\Sigma_{\eta_j}\cap\{1-|z|\le s\}\bigr).
$$
Since $L_0(t)=j$ on $(\eta_{j+1},\eta_j]$, $L_0(t)=k$ on
$(\eta_{k+1},\eta_k]$ for $k>j$, and the total mass is supported in
$(0,\eta_j]$, summation by parts gives
$$
\begin{aligned}
        \int_{\Sigma_{\eta_j}} L_0(1-|z|)\,
        d\omega_a^{\Omega_{\eta_j}}(z)
        &\le
        j\,\nu_{j,a}((0,\eta_j])
        +
        \sum_{k>j}\nu_{j,a}((0,\eta_k]).
\end{aligned}
$$
The first term is bounded by
$$
        j\,p_{R_m}(\eta_j)\le 2^{-j-m}.
$$
For $k>j$, condition~\eqref{eq:qj-condition}, applied at stage $k$ with
$i=j$, gives
$$
        \nu_{j,a}((0,\eta_k])
        \le
        q_{R_m,\eta_j}(\eta_k)
        \le
        2^{-k-j-m}.
$$
Therefore
$$
        \int_{\Sigma_{\eta_j}} L_0(1-|z|)\,
        d\omega_a^{\Omega_{\eta_j}}(z)
        \le
        2^{-j-m}
        +
        \sum_{k>j}2^{-k-j-m},
$$
which tends to $0$ as $j\to\infty$, uniformly for $|a|\le R_m$.

Finally choose a continuous nonincreasing function $L:(0,1]\to[1,\infty)$
such that
$$
       L_0(t)\le L(t)\le L_0(t)+1, \qquad 0<t\leq 1.
$$
This can be done by smoothing the jumps of $L_0$ on the intervals immediately
above the points $\eta_j$.  Then we still have
$
        L(t)\to\infty,$ and
        $tL(t)\to 0$ as $t\to 0$.
Moreover,
$$
\begin{aligned}
        \int_{\Sigma_{\eta_j}} L(1-|z|)\,
        d\omega_a^{\Omega_{\eta_j}}(z)
        &\le
        \int_{\Sigma_{\eta_j}} L_0(1-|z|)\,
        d\omega_a^{\Omega_{\eta_j}}(z)
        +
        \omega_a^{\Omega_{\eta_j}}(\Sigma_{\eta_j}).
\end{aligned}
$$
The first term tends to $0$ uniformly for $|a|\le R_m$, as shown above,
and the second tends to $0$ uniformly by \cref{lem:pR}.  Hence $L$
has the required properties.
\end{proof}

\subsection{Proof of \cref{thm:intro-collar-spike}}
Choose $a\in\D$ with $w(a)>0$, and choose $m$ such that $|a|\le R_m$.
Suppose the conclusion fails. Then there exist $A>0$ and $\eta_0>0$
such that
$$
        w(re^{i\theta})<A L(1-r)
$$
whenever
$$
        0<1-r<\eta_0,
        \qquad
        \dist(e^{i\theta},E)<\eta_0.
$$

Choose $j\ge m$ so large that
$
        \eta_j<\min\{\eta_0,1-|a|\}.
$
Then $a\in\Omega_{\eta_j}$. Moreover, by the definition of
$\Sigma_{\eta_j}$, every $z=re^{i\theta}\in\Sigma_{\eta_j}$ satisfies
$
        0<1-r\le \eta_j$
        and
        $\dist(e^{i\theta},E)\le \eta_j$.
Hence $\Sigma_{\eta_j}$ is contained in the above collar, and therefore
$
        w(z)<A L(1-|z|)
$ for $z\in\Sigma_{\eta_j}$.

Fix $0<s<\eta_j$ and truncate $\Omega_{\eta_j}$ near the unit
circle by considering the component $\Omega_{\eta_j,s}$ of
$$
        \Omega_{\eta_j}\cap\{|z|<1-s\}
$$
which contains $a$.  On the part of the inner collar boundary belonging to
$\partial\Omega_{\eta_j,s}$, we have
$
        w(z) < A L(1-|z|)$.
On the remaining boundary near $\T\setminus E_{\eta_j}$, the function $w$
is uniformly small as $s\to 0$, because
$\T\setminus E_{\eta_j}$ is compactly contained in $\T\setminus E$ and
$w$ has continuous boundary value $0$ there. 
Thus, there exists $\varepsilon_s\to 0$ as $s\to 0$ such that 
\begin{equation} \label{eq:harm1}
        w(a)
        \le
        A\int_{\Sigma_{\eta_j,s}}
        L(1-|z|)
        \,d\omega_a^{\Omega_{\eta_j,s}}(z) + \varepsilon_s,
\end{equation}
by the subharmonic maximum principle in
$\Omega_{\eta_j,s}$. Domain monotonicity of harmonic measure further implies
\begin{equation} \label{eq:harm2}
        \int_{\Sigma_{\eta_j,s}} L(1-|z|)
        \,d\omega_a^{\Omega_{\eta_j,s}}(z)
        \le
        \int_{\Sigma_{\eta_j}} L(1-|z|)
        \,d\omega_a^{\Omega_{\eta_j}}(z).
\end{equation}
Combining \eqref{eq:harm1} and \eqref{eq:harm2} and letting $s\to 0$ gives
$$
        w(a)
        \le
        A\int_{\Sigma_{\eta_j}}
        L(1-|z|)
        \,d\omega_a^{\Omega_{\eta_j}}(z).
$$
Now let $j\to\infty$.  Since $|a|\le R_m$, \cref{lem:weighted-collar}
implies $w(a)\le0$, contradicting $w(a)>0$.  This proves the Theorem.

\section{Proof of Esterle's conjecture}

Let $F,F'$ be Hilbert spaces and let
$\Theta:\C_\infty\setminus E\to\calL(F,F')$ be holomorphic.  We say that
$\Theta$ is unitary-valued on $\T\setminus E$ if $\Theta(\zeta)$ is a
unitary operator from $F$ onto $F'$ for each $\zeta\in\T\setminus E$.
For $n\ge1$, set
$$
        \delta_n(\Theta)=
        \inf_{z\in\D}
        \max\bigl\{|z|^n,\|\Theta(z)^{-1}\|^{-1}\bigr\},
$$
with the convention that $\|\Theta(z)^{-1}\|^{-1}=0$ if $\Theta(z)$ is not
invertible. The quantity $\delta_n$ appeared in \cite{Ransford} as the central quantity to control. In fact, Ransford formulated the following Theorem as a conjecture, except with the requirement that  $\Theta(z)$ is purely contractive in $\D$, meaning that $\|\Theta(z) x\|<\|x\|$ for all $z\in\D$, $x\in  F \setminus\{0\}$.
We only require the weaker assumptions that $\Theta$ is contractive in $\D$ and not constant unitary.

\begin{theorem}\label{thm:operator-uniformity}
Let $E\subset\T$ be closed and have Lebesgue measure zero.  There exists a positive sequence
$\eps_n\to 0$ such that the following holds. Suppose that
$\Theta:\C_\infty\setminus E\to\calL(F,F')$ is holomorphic, contractive in
$\D$, unitary-valued on $\T\setminus E$, and not constant unitary.  Then
$$
        \liminf_{n\to\infty}\frac{\delta_n(\Theta)}{\eps_n}=0.
$$
\end{theorem}

To prove this theorem, we need the following consequence of the work in Section \ref{sec:potential}

\begin{lemma}\label[lemma]{thm:collar-uniformity}
Let $E\subset\T$ be closed and have Lebesgue measure zero.  Then there exists a
positive sequence $\eps_n\to 0$, depending only on $E$, such that
for every nonnegative, not identically zero subharmonic $w$ carried by $E$,
$$
        \liminf_{n\to\infty}\frac{d_n(w)}{\eps_n}=0,
$$
where $$
        d_n(w)
        =
        \inf_{z\in\D}
        \max\bigl\{|z|^n,e^{-w(z)}\bigr\}.
$$
\end{lemma}

\begin{proof}[Proof of Lemma \ref{thm:collar-uniformity}]
Let $L$ be the weight from \cref{lem:weighted-collar}.  Put
$$
        M(t)=\frac{L(t)}{4t}.
$$
Then $M(t)\to\infty$ as $t\to 0$.  For $n\ge1$, define
$$
        \alpha_n=\inf\{L(t):0<t<1,\ M(t)\ge n\}.
$$
The set in the infimum is nonempty. For example, if $t=1/(4n)$, then
$M(t)=nL(t)\ge n$.  
Also $\alpha_n\to\infty$.  Indeed, if
$M(t_j)\ge n_j\to\infty$ while $L(t_j)$ stayed bounded, then
$t_j\to0$, contradicting $L(t)\to\infty$.
Set
$$
        \beta_n=\frac1{10}\sqrt{\inf_{m\geq n}\alpha_m},
        \qquad
        \eps_n=e^{-\beta_n}.
$$
Clearly $\eps_n$ is
nonincreasing and tends to zero.  Moreover $\eps_n=e^{-o(n)}$.  Indeed,
with $t_n=1/(4n)$, the admissibility above gives
$
        \alpha_n\le L(1/(4n)).
$
Since $tL(t)\to0$, this implies $\alpha_n=o(n)$, and hence
$$
        \beta_n\le \frac1{10}\sqrt{\alpha_n}=o(n).
$$

Let $w\ge0$ be nonzero and carried by $E$.  By
\cref{thm:intro-collar-spike} with $A=1$, there are points $z_k\in\D$ such that
$$
        t_k=1-|z_k|\to0,
        \qquad
        w(z_k)\ge L(t_k).
$$
Set
$$
        n_k=\left\lfloor\frac{L(t_k)}{4t_k}\right\rfloor.
$$
Since $M(t_k)\to\infty$, we have $n_k\to\infty$.  Passing to a subsequence
if necessary, assume that $n_k$ is strictly increasing.  For all large $k$,
$
        n_k t_k\geq L(t_k)/8.
$
Therefore
$$
\begin{aligned}
        d_{n_k}(w)
        &\le \max\{|z_k|^{n_k},e^{-w(z_k)}\} \\
        &\le \max\{(1-t_k)^{n_k},e^{-L(t_k)}\} \\
        &\le e^{-cL(t_k)}
\end{aligned}
$$
for an absolute constant $c>0$.  Since $n_k\le M(t_k)$, the point $t_k$
is admissible in the definition of $\alpha_{n_k}$.  Hence
$$
        \alpha_{n_k}\le L(t_k),
        \qquad
        \beta_{n_k}\le \frac1{10}\sqrt{L(t_k)}.
$$
Thus
$$
        \frac{d_{n_k}(w)}{\eps_{n_k}}
        =d_{n_k}(w)e^{\beta_{n_k}}
        \le
        \exp\big(-cL(t_k)+\frac1{10}\sqrt{L(t_k)}\big)
        \longrightarrow0.
$$
The result follows.
\end{proof}

\begin{proof}[Proof of \cref{thm:operator-uniformity}]
If $\Theta(z_0)$ is not invertible for some $z_0\in\D$, then
$\delta_n(\Theta)\le |z_0|^n$.  Since $\eps_n=e^{-o(n)}$, it follows that
$\delta_n(\Theta)/\eps_n\to0$.
Assume therefore that $\Theta(z)$ is invertible for every $z\in\D$, and
put
$$
        w_\Theta(z)=\log\|\Theta(z)^{-1}\|.
$$
Then $w_\Theta$ is subharmonic.  Since $\Theta(z)$ is contractive,
$\|\Theta(z)^{-1}\|\ge1$, and so $w_\Theta\ge0$.  Also, since $\Theta$
extends holomorphically through $\T\setminus E$ and is unitary-valued there,
$w_\Theta$ has continuous boundary value zero on $\T\setminus E$.  Thus
$w_\Theta$ is carried by $E$, unless it is identically zero.

If $w_\Theta\not\equiv0$, then the conclusion follows from \cref{thm:collar-uniformity} since $d_n(w_\Theta)=\delta_n(\Theta)$.
It remains to exclude the case $w_\Theta\equiv0$.  Then
$\|\Theta(z)^{-1}\|=1$ for all $z\in\D$.  Since $\Theta(z)$ is
contractive, for every $x\in F$,
$$
        \|\Theta(z)x\|\le\|x\|,
        \qquad
        \|x\|=\|\Theta(z)^{-1}\Theta(z)x\|
        \le \|\Theta(z)x\|.
$$
Thus $\Theta(z)$ is a surjective isometry for every $z\in\D$, hence
unitary.  For each fixed $x\in F$, the Hilbert-space-valued holomorphic
function $z\mapsto\Theta(z)x$ has constant norm $\|x\|$, and is therefore
constant.  Hence $\Theta$ is constant unitary, contrary to the hypothesis.
\end{proof}


\begin{proof}[Proof of \cref{thm:esterle}]
We now explain how \cref{thm:operator-uniformity} gives \cref{thm:esterle}, following \cite{Ransford}.
Let $(\eps_n)$ be the sequence from Lemma
\ref{thm:collar-uniformity}, and set
$
        u_n=\eps_n^{-1}.
$
Then $u_n\to\infty$.
Let $T$ be a non-unitary Hilbert-space contraction with
$\sigma(T)\subset E$.  We show that
$$
        \limsup_{n\to\infty}\frac{\|T^{-n}\|}{u_n}=\infty.
$$
This is enough to prove the theorem.
By a standard reduction as in Ransford's work, using the Sz.-Nagy--Foias functional-model theory for contractions, one may pass to a completely non-unitary
contraction $T_1$ on a separable Hilbert space such that
$\sigma(T_1)\subset\sigma(T)$ and
$
        \|T_1^{-n}\|\le\|T^{-n}\|$ for $n\geq 1$ \cite[Lemma 6.1]{Ransford}.
Since $\sigma(T_1)\subset\T$ has Lebesgue measure zero, both $T_1^n$ and
$T_1^{*n}$ converge strongly to zero \cite[Lemma 6.2]{Ransford}.  The Sz.-Nagy--Foias model theorem then
realizes $T_1$ as a compressed shift $S_\Theta$ associated with an
operator-valued inner function $\Theta$.  Moreover, $\Theta$ may be chosen
so that it is purely contractive in $\D$, extends holomorphically to
$\C_\infty\setminus E$, and is unitary-valued on $\T\setminus E$. (See \cite[Section 5]{Ransford} for more details.) The hypotheses of \cref{thm:operator-uniformity} are therefore satisfied. It follows that
along a subsequence, $\eps_n/\delta_n(\Theta)\to\infty$.  
Finally, \cite[Theorem 5.5]{Ransford} gives
$$
        \|S_\Theta^{-n}\|
        \ge
        \frac12\Big(\frac1{\delta_n(\Theta)}-1\Big),
        \qquad n\geq 1.
$$
Hence
$$
\begin{aligned}
        \limsup_{n\to\infty}\frac{\|T^{-n}\|}{u_n}
        &\ge
        \limsup_{n\to\infty}\frac{\|T_1^{-n}\|}{u_n} \\
        &=
        \limsup_{n\to\infty}\eps_n\|S_\Theta^{-n}\| \\
        &\ge
        \frac12\limsup_{n\to\infty}
        \Big(\frac{\eps_n}{\delta_n(\Theta)}-\eps_n\Big)
        =\infty,
\end{aligned} 
$$ 
as desired.
\end{proof}

\subsection*{Acknowledgments} The author would like to thank Marcu-Antone Orsoni for his comments, and Thomas Ransford for helpful discussions and for sharing his manuscript \cite{Ransford2}.

\bibliographystyle{acm}
\bibliography{reference.bib}

\end{document}